\documentclass[10pt]{extarticle}

\usepackage{amsmath,amsthm,amssymb,amsfonts}
\usepackage{enumerate}
\usepackage{graphicx}
\usepackage{tikz}

\newtheorem{theorem}{Theorem}
\newtheorem{conjecture}[theorem]{Conjecture}
\newtheorem{example}[theorem]{Example}
\newtheorem{remark}[theorem]{Remark}
\newtheorem{question}[theorem]{Question}
\newtheorem{lemma}[theorem]{Lemma}
\newtheorem{proposition}[theorem]{Proposition}
\newtheorem{corollary}[theorem]{Corollary}
\newtheorem{definition}[theorem]{Definition}

\newcommand{\RR}{\mathbb{R}}
\newcommand{\NN}{\mathbb{N}}

\newcommand{\ZZ}{\mathbb{Z}}

\newcommand{\Tn}{\mathcal{U}_N}

\DeclareMathOperator*{\argmin}{arg\,min}
\DeclareMathOperator*{\argmax}{arg\,max}
\DeclareMathOperator*{\peak}{peak}
\DeclareMathOperator*{\valley}{valley}
\newcommand{\aaa}{{\bf a}}
\newcommand{\vv}{{\bf v}}
\newcommand{\uu}{{\bf u}}
\newcommand{\ww}{{\bf w}}

\begin{document}
	
	\title{Tropical Fermat-Weber points}
	\author{Bo Lin and Ruriko Yoshida}
	\date{}
	\maketitle

\begin{abstract}

In a metric space, the Fermat-Weber points of a sample are statistics to measure the central
tendency of the sample and it is well-known that
the Fermat-Weber point of a sample is not necessarily unique in the metric space.
We investigate the computation of Fermat-Weber points under the
tropical metric on the quotient 
space $\mathbb{R}^{n} \!/ \mathbb{R} {\bf 1}$ with a fixed $n \in \NN$,
motivated by its application to the space of equidistant phylogenetic
trees with $N$ leaves (in this case $n=\binom{N}{2}$) realized as the tropical linear space of all
ultrametrics. We show that the set of all tropical Fermat-Weber points of 
a finite sample is always a classical convex polytope, and we present a combinatorial formula for a key value 
associated to this set.
We identify conditions under which this set is a singleton. We apply numerical experiments to
analyze the set of the tropical Fermat-Weber points within a space of phylogenetic trees. We discuss 
the issues in the computation of the tropical Fermat-Weber points.

\end{abstract}

\section{Introduction}

The Fr\'{e}chet mean and the Fermat-Weber point of a sample are statistics to measure the central tendency of the sample \cite{Cieslikb,Nielsen}. For any metric space with a distance metric $d(\cdot, \cdot)$ between any two points, the Fr\'{e}chet population mean of a distribution $\nu$ is defined as follows:
\begin{equation*}
\mu = \argmin_{{\bf y}} \int d({\bf y},{\bf x})^2 d\nu({\bf x}), 
\end{equation*}
and thus, the Fr\'{e}chet sample mean of a sample ${\bf x}_1, \ldots, {\bf x}_m$ is defined as
\begin{equation*}
\hat{\mu}  = \argmin_{{\bf y}} \sum_{i=1}^m d({\bf y},{\bf x}_i)^2
\end{equation*}
The Fermat-Weber point of a distribution $\nu$ is defined as follows:
\begin{equation*}
\mu = \argmin_{{\bf y}} \int d({\bf y},{\bf x}) d\nu({\bf x}),
\end{equation*}
and thus,  the Fermat-Weber point of a sample is defined as follows:
\begin{equation*}
\hat{\mu}  = \argmin_{{\bf y}} \sum_{i=1}^m d({\bf y},{\bf x}_i).
\end{equation*}

In this paper, we consider tropical Fermat-Weber points on the quotient 
space $\mathbb{R}^{n} \!/ \mathbb{R} {\bf 1}$ for a fixed positive
integer $n$, that is, Fermat-Weber points on the quotient 
space $\mathbb{R}^{n} \!/ \mathbb{R} {\bf 1}$ under the 'tropical
metric' in the max-plus algebra. 
This tropical metric is called the {\em generalized Hilbert projective metric} 
\cite[\S 2.2]{AGNS}, \cite[\S 3.3]{CGQ}.  It is also known that the
geodesic between two points under this metric may not be unique. For more details on this
metric and tropical geometry, see \cite{sturmfels}.
 
Here we focus on the computational aspect of the Fermat-Weber
points in $\mathbb{R}^{n} \!/ \mathbb{R} {\bf 1}$ under the tropical
metric in the max-plus algebra, including characterizing the set of
all tropical Fermat-Weber points in $\mathbb{R}^{n} \!/ \mathbb{R}
{\bf 1}$.  More specifically, 
in Section \ref{key:sum}, we show an important
property of a tropical Fermat-Weber point of $m$ points ${\bf v}_1, \ldots , {\bf v}_m$
over $\RR^n/{\RR \bf 1}$ in Theorem \ref{thm:sum}. Then in Proposition \ref{prop:polytope} we show that the set of all tropical Fermat-Weber points of $m$ points ${\bf v}_1, \ldots , {\bf v}_m$
in $\RR^n/{\RR \bf 1}$ forms a classical convex polytope. Therefore, there are many cases when a set of $m$ points has infinitely many tropical Fermat-Weber points. In Section
\ref{sec:unique}, we investigate the condition when a random sample
of $m$ points has a unique tropical Fermat-Weber point.  If we 
consider the space of families of $m$ arbitrary points in $\RR^n/{\RR \bf 1}$ (which corresponds to $\RR^{m(n-1)}$),
the points forming an {\em essential set} (Definition
\ref{def:ess}) with a unique tropical Fermat-Weber point are contained
in a finite union of proper linear subspaces in $\RR^{m(n-1)}$ (Theorem
\ref{thm:lowdim}).  This theorem implies
that if we pick a random sample of $m$ points
${\bf v}_1,{\bf v}_2,\ldots,{\bf v}_m$ in $\RR^{m(n-1)}$, then with
probability $1$ we get either a set of
points such that one of them is already a unique tropical Fermat-Weber point of the others, or
a set of points that has infinitely many tropical Fermat-Weber points.

One finds an application of tropical Fermat-Weber points in
phylogenomics.  
In recent decades, the field of phylogenetics has found its
applications in analysis on genomic scale data (phylogenomics).  In particular, it has
been applied to analyze the relations between species and populations,
genome evolution, as well as evolutionary processes of 
speciation and molecular evolution. Today, since we can generate genomic data 
relatively cheaply and quickly, we encounter a new problem in the
sheer volume of genomic data and we lack analytical tools on such data 
(e.g. \cite{Betancur2013,Carling2008,Heled2013,Thompson2013,Yu2011b}). 
Lin et.~al \cite{LSTY} mentioned tropical Fermat-Weber points on $\Tn$, the treespace of
rooted equidistant phylogenetic trees with $N$ leaves as a possible statistical method
to summarize genome data sets. 
Therefore, in Section \ref{sec:tree},  we investigate the intersection between the set of tropical
Fermat-Weber points of the sample in $\mathbb{R}^{{N \choose 2}} \!/
\mathbb{R} {\bf 1}$ and $\Tn$ for small $N \in \NN$. We show by
experiments that it is very rare to obtain a unique tropical Fermat-Weber point
of an essential random sample over $\Tn$.  From our
experimental study, we conjecture that if an essential random
sample has a unique tropical Fermat-Weber point, then the
unique tropical Fermat-Weber point is the vector with all ones in its
coordinates. 

In Section \ref{sec:elli} we generalize the locus of the tropical Fermat-Weber
points of a sample of size $k$ to the $k$-ellipses under tropical
metric.  We end this paper with an open problem regarding the computation of the tropical Fermat-Weber points of a sample.

\section{The tropical Fermat-Weber point}\label{key:sum}

In this section, we define the tropical metric and derive some basic properties
of the tropical Fermat-Weber point. 

For ${\bf u}, {\bf v}\in \RR^n$ we define their distance as
\begin{equation}\label{equ:met}
d_{tr}({\bf u}, {\bf v})=\max_{1\le i<j\le n}\{|u_i-u_j-v_i+v_j|\}.
\end{equation}
In other words, let $\mathcal{D}_{\uu,\vv}=\{u_i-v_i|1\le i\le n\}$, then
\begin{equation}\label{equ:difset}
d_{tr}({\bf u}, {\bf v})=\max_{x,y\in \mathcal{D}_{\uu,\vv}}{(|x-y|)}=\max(\mathcal{D}_{\uu,\vv})-\min(\mathcal{D}_{\uu,\vv}).
\end{equation}
By definition $d_{tr}$ is reflexive. For any ${\bf u}, {\bf v}, {\bf w}\in \RR^n$, we have that
\begin{eqnarray*}
	d_{tr}({\bf u}, {\bf w})&=&\max_{1\le i<j\le n}\{|u_i-u_j-w_i+w_j|\} \\
	&=&\max_{1\le i<j\le n}\{|(u_i-u_j-v_i+v_j)+(v_i-v_j-w_i+w_j)|\} \\
	&\le &\max_{1\le i<j\le n}\{(|u_i-u_j-v_i+v_j|+|v_i-v_j-w_i+w_j|)\} \\
	&\le &\max_{1\le i<j\le n}\{|u_i-u_j-v_i+v_j|\}+\max_{1\le i<j\le n}\{|v_i-v_j-v_i+v_j|\}\\
	&=&d_{tr}({\bf u}, {\bf v})+d_{tr}({\bf v}, {\bf w}).
\end{eqnarray*}
Thus, $d_{tr}$ satisfies the triangle inequality. Note that $d_{tr}({\bf u}, {\bf v})=0$ if and only if ${\bf u}-{\bf
  v}$ is a scalar multiple of ${\bf 1}$, and for any scalar multiple
$c{\bf 1}$ with a constant $c \in \RR$, $d_{tr}({\bf u}+c{\bf 1},{\bf
  v})=d_{tr}({\bf u}, {\bf v})$. So $d_{tr}({\bf u}, {\bf v})=0$ if and only if ${\bf u}={\bf v}$ in the
quotient space $\RR^n/\RR{\bf 1}$. Then $d_{tr}$ is a \emph{metric} on $\RR^n/\RR{\bf 1}$. It is called the \emph{tropical
 metric} \cite{AGNS}.

\begin{remark}
	The metric $d_{tr}$ is invariant under vector addition in Euclidean space: for any ${\bf u}, {\bf v}, {\bf w}\in \RR^n/\RR{\bf 1}$, by (\ref{equ:difset}), we have $d_{tr}({\bf u}+{\bf w},{\bf v}+{\bf w})=d_{tr}({\bf u},{\bf v})$.
\end{remark}

Given vectors $\vv_1,\ldots, \vv_m\in \mathbb{R}^n/{\RR \bf 1}$, the set of their
\emph{tropical Fermat-Weber points} (if the context is clear, we 
simply use \emph{Fermat-Weber points}) is
\begin{equation}\label{equ:sum}
\argmin_{\uu\in \RR^n/{\RR \bf 1}} \mathop{\sum}_{i=1}^{m}{d_{tr}(\uu,\vv_{i})}.
\end{equation}

\begin{definition}
	For points $\vv_1,\vv_2,\ldots,\vv_m \in \RR^n/{\RR \bf 1}$, we define the minimal sum of distances from them as 
	\begin{equation}\label{equ:sumofdist}
	{\bf d}(\vv_1,\vv_2,\ldots,\vv_m)=\min_{\uu\in \RR^n/{\RR \bf 1}}{\mathop{\sum}_{i=1}^{m}{d_{tr}(\uu,\vv_{i})}}.
	\end{equation}
\end{definition}

Then ${\bf d}(\vv_1,\vv_2,\ldots,\vv_m)$ should be determined by the entries
of $\vv_1,\vv_2,\ldots,\vv_m$. However, at this point we do not know
whether ${\bf d}(\vv_1,\vv_2,\ldots,\vv_m)$ is well-defined, nor any explicit formulation of it. In addition, in order to find the set of Fermat-Weber points of $\vv_1,\vv_2,\ldots,\vv_m$, we need to know the value of ${\bf d} (\vv_1,\vv_2,\ldots,\vv_m)$. The following theorem gives a direct formula of ${\bf d} (\vv_1,\vv_2,\ldots,\vv_m)$.
\begin{theorem}\label{thm:sum}
	Let $M$ be an $m\times n$ matrix with real entries such that the row vectors are $\vv_1,\vv_2,\ldots,\vv_m$. Then
	\begin{equation}\label{equ:thm}
	{\bf d} (\vv_1,\vv_2,\ldots,\vv_m)=\max_{\sigma,\tau}{\left|\sum_{i=1}^{m}{M_{i,\sigma(i)}}-\sum_{i=1}^{m}{M_{i,\tau(i)}}\right|},
	\end{equation}
	where functions $\sigma,\tau: [m]\to [n]$ satisfy $\sigma([m])=\tau([m])$ as multisets.
\end{theorem}

To prove this theorem, we need the following lemmas.
\begin{lemma}\label{lem:ine}
	The right hand side (RHS) of (\ref{equ:thm}) is bounded above
        by the left hand side (LHS) of (\ref{equ:thm}).
\end{lemma}

\begin{proof}
	Let $\uu$ be a Fermat-Weber point of $\vv_1,\vv_2,\ldots,\vv_m$. Suppose $\sigma,\tau$ are functions with the same multiset of values. Since $\sigma$ and $\tau$ are symmetric, we may assume that 
	\begin{equation*}\sum_{i=1}^{m}{M_{i,\sigma(i)}}\ge \sum_{i=1}^{m}{M_{i,\tau(i)}}.\end{equation*}
	Now for $1\le i\le m$ we have
	\begin{equation*}d_{tr}(\uu,\vv_{i})\ge |(\vv_{i})_{\sigma(i)}-(\vv_{i})_{\tau(i)}-u_{\sigma(i)}+u_{\tau(i)}|\ge M_{i,\sigma(i)}-M_{i,\tau(i)}-u_{\sigma(i)}+u_{\tau(i)}.\end{equation*}
	Summing up over $1\le i\le m$, the LHS of the sum is ${\bf d} (\vv_1,\vv_2,\ldots,\vv_m)$. One part of the RHS of the sum is 
	\begin{equation*}\sum_{i=1}^{m}{M_{i,\sigma(i)}}-\sum_{i=1}^{m}{M_{i,\tau(i)}},\end{equation*}
	and the other part vanishes because $\sigma([m])=\tau([m])$. Hence ${\bf d} (\vv_1,\vv_2,\ldots,\vv_m)\ge \sum_{i=1}^{m}{M_{i,\sigma(i)}}-\sum_{i=1}^{m}{M_{i,\tau(i)}}$.
\end{proof}

\begin{lemma}\label{lem:mat}
	If $A$ and $B$ are two $m\times n$ matrices that have the same
        multiset of entries, then there exist $m\times n$ matrices $A'$ and $B'$ such that: 
	\begin{enumerate}[(i)]
		\item \label{con:row} for $1\le i\le m$, the entries of the $i$-th row of $A'$ and the entries of the $i$-th row of $B'$ form the same multiset; and
		\item \label{con:col} for $1\le j\le n$, the entries of the $j$-th column of $A'$ and the entries of the $j$-th column of $A$ form the same multiset, and the entries of the $j$-th column of $B'$ and the entries of the $j$-th column of $B$ form the same multiset.
	\end{enumerate}
\end{lemma}

\begin{proof}
	If the entries of $A$ are not all distinct, then we can label the equal entries to distinguish them. So we may assume that both $A$ and $B$ have $mn$ distinct entries. Then we can replace \emph{multiset} in the statement by \emph{set}.
	
	We use induction on $m$. If $m=1$, we can take $A'=A$ and $B'=B$. Suppose Lemma \ref{lem:mat} is true when $m\le k$, we consider the case when $m=k+1$. If there exists $m\times n$ matrices $A''$ and $B''$ such that (\ref{con:col}) is true and (\ref{con:row}) is true for $i=1$, let $r_{A}$ and $r_{B}$ be the vector of first row in $A''$ and $B''$ respectively, and we denote
	\begin{equation*}
		A''= \begin{bmatrix}
		r_{A} \\
		A_{2}
		\end{bmatrix}, \,\, 
		B''= \begin{bmatrix}
		r_{B} \\
		B_{2}
		\end{bmatrix}.
	\end{equation*}
	Then we apply the induction hypothesis of $m=k$ to the matrices $A_2$ and $B_2$. Suppose we get new matrices $A'_2$ and $B'_2$ respectively, then we let
	\begin{equation*}
	A'= \begin{bmatrix}
	r_{A} \\
	A'_{2}
	\end{bmatrix}, \,\, 
	B'= \begin{bmatrix}
	r_{B} \\
	B'_{2}
	\end{bmatrix}.
	\end{equation*}
    So $A'$ and $B'$ satisfy both (\ref{con:row}) and (\ref{con:col}). Now it suffices to show that we can find such a pair of matrices $A''$ and $B''$. We denote by $s$ the set of entries in $r_{A}$ (and also the set of entries in $r_{B}$). Then the above claim is equivalent to the following statement: there exists a set $s$ with $|s|=n$ and $s$ has exactly one element in each column of $A$ and $B$.
	
	We construct a bipartite graph $G=(V,E)$, where the two parts of $V$ correspond to the columns of $A$ and the columns of $B$ respectively:
	\begin{equation*}V=\{a_1, a_2,\cdots, a_n,
          b_1, b_2,\cdots, b_n\}.\end{equation*}
	For each entry $x$ in the set of entries of $A,B$, if $x$ is in the $i$-th column of $A$ and in the $j$-th column of $B$, then we connect an edge between $a_i$ and $b_j$. So $|E|=mn$. Since each column has $m$ entries, the graph $G$ is $m$-regular. By Hall's Theorem \cite{Hall}, $G$ admits a perfect matching. Then we let $s$ be the set of $n$ elements corresponding to the edges in this perfect matching. The induction step is done. 
\end{proof}

For convenience, if $\sigma: [m]\to [n]$ is a function, then we view it as a vector in $[n]^{m}$ and we define a vector $\ww_{\sigma}\in \mathbb{N}^{n}$ as follows: the $i$-th entry of $\ww_{\sigma}$ is $|\sigma^{-1}(i)|$. So the sum of entries in $\ww_{\sigma}$ is always $m$. For example, if $m=3,n=5$ and $\sigma(1)=4,\sigma(2)=3,\sigma(3)=3$, then $\ww_{\sigma}=(0,0,2,1,0)$.

\begin{proof}[Proof of Theorem \ref{thm:sum}]
	Let $M$ be the value of the RHS in (\ref{equ:thm}). By Lemma \ref{lem:ine}, it suffices to show that there exists a point $\uu=(u_1,\ldots,u_n)\in \RR^n/{\RR \bf 1}$ such that
	
	\begin{equation*}
		\sum_{i=1}^{m}{d_{tr}(\uu,\vv_i)}=M.
	\end{equation*}
	
	For convenience, we introduce parameters $c_i$ to represent $d_{tr}(\uu,\vv_i)$ and another parameter $s$ to represent their sum. Then
	
	\begin{equation}\label{ine:1dis}
		c_i\ge u_j-u_k-M_{i,j}+M_{i,k} \quad \forall 1\le j,k\le n.
	\end{equation}
	
	and
	
	\begin{equation}\label{equ:M-defn}
		s=\sum_{i=1}^{m}{c_i}.
	\end{equation}
	
	Equivalently, we can eliminate the parameters $c_i$ and we get the following family of inequalities:
	\begin{equation}\label{equ:all}
	s\ge \sum_{i=1}^{m}{M_{i,\sigma(i)}}-\sum_{i=1}^{m}{M_{i,\tau(i)}}-{\bf u}\cdot \ww_{\tau}+{\bf u}\cdot \ww_{\sigma}, \forall \sigma,\tau \in [n]^{m}.
	\end{equation}
	
	In other words, (\ref{ine:1dis}) and (\ref{equ:M-defn}) are
        simultaneously feasible if and only if (\ref{equ:all}) is
        feasible. Now it suffices to show that there exists real
        numbers $u_1,\ldots,u_n,s$ satisfying (\ref{equ:all}) and
        $s\le M$. Note that by applying the Fourier-Motzkin Elimination \cite{Zie95}, one may get rid of the variables $u_i$ one at a time. After finite steps, the only variable remaining in the inequalities is $s$. Apparently $s$ can be arbitrarily large, so all remaining inequalities are of the form $s\ge l$, where the lower bound $l$ is a constant. Then $s$ could be the maximum of these lower bounds $c$, and it suffices to show that any lower bound of $s$ we obtain from the Fourier-Motzkin Elimination is at most $M$.
	
	Note that in each step of Fourier-Motzkin Elimination, we obtain a new inequality as a $\mathbb{Q_{}+}$-linear combination of existing inequalities. Therefore if $l$ is a lower bound, then $s\ge l$
	is a $\mathbb{Q_{}+}$-linear combination of the inequalities in (\ref{equ:all}). Multiplying by a positive integer we may assume that it's a $\mathbb{Z_{}+}$-linear combination, therefore we have $r\in \mathbb{Z_{}+}$ and functions $\sigma_1,\ldots,\sigma_r,\tau_1,\ldots,\tau_r: [m]\to [n]$ such that $s\ge l$ is equivalent to
	\begin{equation}\label{equ:com}
	rs\ge \sum_{j=1}^{r}{\left(\sum_{i=1}^{m}{M_{i,\sigma_{j}(i)}}-\sum_{i=1}^{m}{M_{i,\tau_{j}(i)}}-{\bf u}\cdot \ww_{\tau_{j}}+{\bf u}\cdot \ww_{\sigma_{j}}\right)}.
	\end{equation}
	So the RHS of (\ref{equ:com}) is the constant $rl$. Therefore
	\begin{equation}\label{equ:low}
	rl=\sum_{j=1}^{r}{\left(\sum_{i=1}^{m}{M_{i,\sigma_{j}(i)}}-\sum_{i=1}^{m}{M_{i,\tau_{j}(i)}}\right)},
	\end{equation}
	and
	\begin{equation}\label{equ:inn}
	\sum_{j=1}^{r}{\ww_{\sigma_{j}}}=\sum_{j=1}^{r}{\ww_{\tau_{j}}}.
	\end{equation}
	
	Now let $A$ and $B$ be two matrices with the same size $r\times m$, with $A_{j,i}=\sigma_{j}(i)$ and $B_{j,i}=\tau_{j}(i)$ for $1\le j\le r, 1\le i\le m$. Then the multiset of their entries are equal because of (\ref{equ:inn}). By Lemma \ref{lem:mat} we can obtain matrices $A'$ and $B'$ satisfying the conditions in (\ref{con:row}) and (\ref{con:col}). For $1\le j\le r$, we let $\sigma'_{j}$ and $\tau'_{j}$ be functions mapping from $[m]$ to $[n]$ such that for $1\le i\le m$, $\sigma'_{j}(i)=A'_{j,i}$ and $\tau'_{j}(i)=B'_{j,i}$
    Then the condition (\ref{con:row}) implies that 
	\begin{equation*}\ww_{\sigma'_{j}}=\ww_{\tau'_{j}} \quad \forall 1\le j\le r.\end{equation*}
	The condition (\ref{con:col}) implies that for each $1\le i\le m$, the multisets $\{\sigma_{j}(i)|1\le j\le r\}$ and $\{\sigma'_{j}(i)|1\le j\le r\}$ are equal and the multisets $\{\tau_{j}(i)|1\le j\le r\}$ and $\{\tau'_{j}(i)|1\le j\le r\}$ are equal. Then
	\begin{equation}\label{equ:sigma}
	\sum_{j=1}^{r}{\sum_{i=1}^{m}{M_{i,\sigma_{j}(i)}}}=\sum_{j=1}^{r}{\sum_{i=1}^{m}{M_{i,\sigma'_{j}(i)}}}
	\end{equation}
	because each entry of $M$ is added by the same number of times in both sides of (\ref{equ:sigma}). Similarly
	\begin{equation}\label{equ:tau}
		\sum_{j=1}^{r}{\sum_{i=1}^{m}{M_{i,\tau_{j}(i)}}}=\sum_{j=1}^{r}{\sum_{i=1}^{m}{M_{i,\tau'_{j}(i)}}}.
	\end{equation}
	Then the inequality (\ref{equ:com}) is equivalent to
	\begin{equation}\label{equ:std}
	rs\ge \sum_{j=1}^{r}{(\sum_{i=1}^{m}{M_{i,\sigma'_{j}(i)}}-\sum_{i=1}^{m}{M_{i,\tau'_{j}(i)}}-{\bf u}\cdot \vv_{\tau'_{j}}+{\bf u}\cdot \vv_{\sigma'_{j}})}.
	\end{equation}
	Now for each $1\le j\le r$, 
	\begin{equation*}\sum_{i=1}^{m}{M_{i,\sigma'_{j}(i)}}-\sum_{i=1}^{m}{M_{i,\tau'_{j}(i)}}-{\bf u}\cdot \vv_{\tau'_{j}}+{\bf u}\cdot \vv_{\sigma'_{j}}=\sum_{i=1}^{m}{M_{i,\sigma'_{j}(i)}}-\sum_{i=1}^{m}{M_{i,\tau'_{j}(i)}}.\end{equation*}
	In addition $\sigma'_j$ and $\tau'_j$ have the same multiset
        of values. Then by the definition of $M$, 
	\begin{equation*}M\ge \sum_{i=1}^{m}{M_{i,\sigma'_{j}(i)}}-\sum_{i=1}^{m}{M_{i,\tau'_{j}(i)}}.\end{equation*}
	We sum over $1\le j\le r$, then
	\begin{equation*}rl = \sum_{j=1}^{r}{(\sum_{i=1}^{m}{M_{i,\sigma'_{j}(i)}}-\sum_{i=1}^{m}{M_{i,\tau'_{j}(i)}})} \le \sum_{j=1}^{r}{M}=rM,\end{equation*}
	hence $l\le M$.
	So $M$ is the greatest possible lower bound of $s$, which means $s=M$ would make the system of linear inequalities feasible. So ${\bf d} (v_1,v_2,\ldots,v_m)=M$.
\end{proof}

\begin{proposition}\cite[Proposition 6.1]{LSTY}\label{prop:polytope}
	Let $\vv_1, \vv_2,\ldots,\vv_m$ be points in $\RR^n/{\RR \bf 1}$. The set of 
	their Fermat-Weber points is a classical convex polytope in $\RR^{n-1}\simeq \RR^n/{\RR \bf 1}$.
\end{proposition}

\begin{proof}[Proof]
	Let ${\bf x}=(x_1,\ldots,x_n)$ be a point in $\RR^n/{\RR \bf 1}$. Then ${\bf x}$ is a Fermat-Weber point of $\vv_1, \vv_2,\ldots,\vv_m$ if and only if for all choices of indices $j_i,k_i\in [n], 1\le i\le m$,
	\begin{equation}\label{ine:lin}
		\sum_{i=1}^{m}{(x_{j_i}-x_{k_i}+v_{i,k_i}-v_{i,j_i})}\le {\bf d} (\vv_1, \vv_2,\ldots,\vv_m).
	\end{equation}
	Then the set is a polyhedron in $\RR^n$. Finally in $\RR^n/{\RR \bf 1}$ we may assume $x_1=0$, and thus $x_i$ is bounded for $2\le i\le n$.
\end{proof}

\begin{example}\label{example:threepoints}
	The polytope of the following three points in $\RR^3/{\RR \bf 1} \simeq \RR^2$
	\[(0, 0, 0), (0, 3, 1), (0, 2, 5)\]
	is the triangle with vertices 
	\[(0,1,1),(0,2,1),(0, 2,2).\]
In Figure \ref{fig:1}, we draw the coordinates $x_2$ and $x_3$ since
the first coordinate $x_1 = 0$.
	\begin{figure}[!h]
		\centering
		\begin{tikzpicture}
			\draw [black,fill=lightgray] (1,1) -- (2,1) -- (2,2) -- (1,1);
			\filldraw [black] (1,1) circle (1pt);
			\filldraw [black] (2,1) circle (1pt);
			\filldraw [black] (2,2) circle (1pt);
			\filldraw [black] (0,0) circle (2pt);
			\filldraw [black] (3,1) circle (2pt);
			\filldraw [black] (2,5) circle (2pt);
			\node [right] at (0,0) {$(0,0)$};
			\node [right] at (3,1) {$(3,1)$};
			\node [below] at (2,5) {$(2,5)$};
		\end{tikzpicture}
		\caption{The Fermat-Weber points of three points in Example \ref{example:threepoints} is a closed triangle (blue).}\label{fig:1}
	\end{figure}
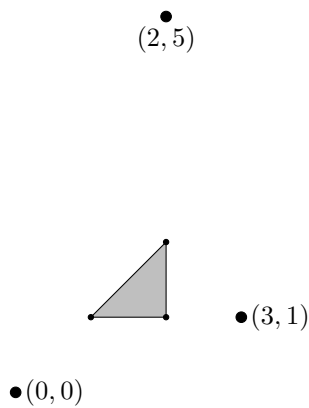
\end{example}

\section{Uniqueness of a Fermat-Weber point under the tropical metric}\label{sec:unique}

In the previous section we have shown that in some cases,
there are infinitely many Fermat-Weber points of a given set of $m$
points in $\RR^n/{\RR \bf 1}$. But how often does this case happen?  In this section we investigate conditions on the set of points in
$\RR^n/{\RR \bf 1}$ that has a unique Fermat-Weber point in $\RR^n/{\RR \bf 1}$,
i.e., we study when a random sample of $m$ points in $\RR^n/{\RR \bf 1}$ has a
unique Fermat-Weber point in $\RR^n/{\RR \bf 1}$.

\begin{lemma}\label{lem:psuni}
	Let $\vv_1,\vv_2,\ldots,\vv_m$ be points in $\RR^n/{\RR \bf 1}$ and $\vv_0$ be a Fermat-Weber point of them. Then $\vv_0,\vv_1,\ldots,\vv_m$ have a unique Fermat-Weber point, which is $\vv_0$.
\end{lemma}

\begin{proof}
	For any point ${\bf x}\in \RR^n/{\RR \bf 1}$, suppose ${\bf x}$ and $\vv_0$ are not the same point in $\RR^n/{\RR \bf 1}$. Then we have
	\begin{equation}
	d_{tr}({\bf x},\vv_0)>0=d_{tr}(\vv_0,\vv_0).
	\end{equation}
	Since $\vv_0$ is a Fermat-Weber point of $\vv_1,\vv_2,\ldots,\vv_m$, we have
	\begin{equation}
	\mathop{\sum}_{i=1}^{m}{d_{tr}({\bf x},\vv_{i})}\ge \mathop{\sum}_{i=1}^{m}{d_{tr}(\vv_0,\vv_{i})}.
	\end{equation}
	So
	\begin{equation}
	\mathop{\sum}_{i=0}^{m}{d_{tr}({\bf x},\vv_{i})}>\mathop{\sum}_{i=0}^{m}{d_{tr}(\vv_0,\vv_{i})}.
	\end{equation}
	Hence, by definition, $\vv_0$ is the unique Fermat-Weber point in $\RR^n/{\RR \bf 1}$.
\end{proof}

The situation in Lemma \ref{lem:psuni} is not desirable, because we don't know whether $\vv_1,\vv_2,\ldots,\vv_m$ have a unique Fermat-Weber point in $\RR^n/{\RR \bf 1}$. So we introduce the following definition.

\begin{definition}\label{def:ess}
	Let $S=\{\vv_1,\vv_2,\ldots,\vv_m\}$ be a set of points in $\RR^n/{\RR \bf 1}$. The set $S$ is \emph{essential} if for $1\le i\le m$, the point $\vv_i$ is not a Fermat-Weber point of the points in $S-\{\vv_i\}$.
\end{definition}

Now we consider the following question: in $\mathbb{R}^n/{\RR \bf 1}$, what is
the smallest integer $u(n)$ such that there exist an essential set of $u(n)$ points with a unique Fermat-Weber point in $\RR^n/{\RR \bf 1}$?

\begin{proposition}\label{prop:un}
	For $n\ge 3$, $u(n)\le n$.
\end{proposition}

\begin{proof}
	First we suppose $n\ge 4$. Then we claim that the row vectors in the following $n\times n$ matrix $M$ represent $n$ points $\vv_1,\cdots,\vv_n$ that form an essential set and have a unique Fermat-Weber point in $\RR^n/{\RR \bf 1}$. 
	\begin{equation*}
	M_{i,j}=\begin{cases}
	1, &\text{ if } j-i\equiv 0,1 \mod n;\\
	-1, &\text{ if } j-i\equiv 2,3 \mod n;\\
	0, &\text{ otherwise}.
	\end{cases}
	\end{equation*}
	
	Note that for $1\le i\le n$ we have	$d_{tr}(\vv_i,{\bf 0})=1-(-1)=2$. Thus,
	\begin{equation}\label{equ:ori}
	\mathop{\sum}_{i=1}^{n}{d_{tr}(\vv_i,{\bf 0})}=2n.
	\end{equation} 
	
	Now suppose ${\aaa}=(a_1,\cdots,a_n)\in \RR^n/{\RR \bf 1}$ is a Fermat-Weber point of $\vv_1,\cdots,\vv_n$. For convenience we denote that $a_{i+n}=a_{i}$ for all $i$. By (\ref{equ:difset}), for $1\le i\le n$ we have
	\begin{equation}\label{equ:dif}
	d_{tr}(\aaa,\vv_i)=\max_{1\le j\le n}\{a_j-M_{i,j}\}-\min_{1\le j\le n}\{a_j-M_{i,j}\}.
	\end{equation}
	Then for $1\le i\le n$, 
	note that the $i$-th and $(i+1)$-th coordinates of $\vv_i$ are $1$. Thus, 
	$$\max_{1\le j\le n}\{a_j-M_{i,j}\}\ge 1-\min\{a_{i},a_{i+1}\}.$$
	Similarly, since the $(i+2)$-th and $(i+3)$-th coordinates of $\vv_i$ are $-1$, 
	$$\min_{1\le j\le n}\{a_j-M_{i,j}\}\le -1-\max\{a_{i+2},a_{i+3}\}.$$
	Then we have
	\begin{eqnarray*}
		d_{tr}(\vv_i, \aaa)&\ge & (1-\min\{a_{i},a_{i+1}\})-(-1-\max\{a_{i+2},a_{i+3}\})\\
		&=& 2+\max\{a_{i+2},a_{i+3}\}-\min\{a_{i},a_{i+1}\}.
	\end{eqnarray*}
	Summing over $i$, we get
	\begin{equation}\label{equ:a}
	\mathop{\sum}_{i=1}^{n}{d_{tr}(\vv_i,{\bf a})}\ge 2n+\mathop{\sum}_{i=1}^{n}{\left[\max\{a_{i},a_{i+1}\}-\min\{a_{i},a_{i+1}\}\right]}\ge 2n.
	\end{equation}
	By (\ref{equ:ori}) and (\ref{equ:a}), we know that $\bf 0$ is a
	Fermat-Weber point. Since ${\bf a}$ is also a Fermat-Weber
	point, all equalities in (\ref{equ:a}) hold. Hence
	$\max\{a_{i},a_{i+1}\}=\min\{a_{i},a_{i+1}\}$ for all $i$, which
	means $a_i=a_{i+1}$ for all $i$. So ${\bf a}={\bf 0}$ in
	$\RR^n/{\RR \bf 1}$. Then $\vv_1,\cdots,\vv_n$ have a unique
	Fermat-Weber point in $\RR^n/{\RR \bf 1}$. Finally since 
	$\vv_i \not = \bf 0$ for each $i = 1, \ldots, n$ in
	$\RR^n/{\RR \bf 1}$, the set of points $\vv_1,\cdots,\vv_n$ forms an essential set.
	
	As for the case when $n=3$, we have the following example of three points in $\RR^3/{\RR \bf 1}$:
	\[(-1, 1, 1), (1, -1, 1), (1, 1, -1).\]
	By simple computation we get that they have a unique Fermat-Weber point $(0,0,0)$ in $\RR^3/{\RR \bf 1}$ and thus they form an essential set.
\end{proof}

Proposition \ref{prop:un} shows the existence of essential sets of points with a unique Fermat-Weber point. However, the following theorem tells us that this case is very rare.

\begin{theorem}\label{thm:lowdim}
	Fix positive integers $m$ and $n$. Consider the space $\RR^{m(n-1)}$ of $m$ points $\vv_1, \vv_2, \ldots, \vv_m$ in $\RR^n/{\RR \bf 1}$. Then the points representing an essential set of points with a unique Fermat-Weber point are contained 
    in a finite union of proper linear subspaces in $\RR^{m(n-1)}$.
\end{theorem}

\begin{definition}\label{def:pkvl}
Let ${\bf u}=(u_1,\ldots,u_n),{\bf v}=(v_1,\ldots,v_n)$ be two points in $\RR^n/{\RR \bf 1}$ and $d=d_{tr}({\bf u,v})>0$. The \emph{peaks} and \emph{valleys} of ${\bf u,v}$ are the following subsets of $[n]$:
\begin{equation*}\peak({\bf u,v})=\argmax_{1\le i\le n}{\{u_i-v_i\}} \text{ , } \valley({\bf u,v})=\argmin_{1\le i\le n}{\{u_i-v_i\}}.\end{equation*}.
\end{definition}

We prove a few lemmas before we prove Theorem \ref{thm:lowdim}.
\begin{lemma}\label{lem:vari}
	Let ${\bf u}=(u_1,\ldots,u_n),{\bf v}=(v_1,\ldots,v_n)$ be two points in $\RR^n/{\RR \bf 1}$ and $d=d_{tr}({\bf u,v})>0$. Let $\epsilon$ be a positive real number less than the minimum of the set
	\begin{equation*}\{|(u_i-v_i)-(u_j-v_j)|:1\le i<j\le n\}-\{0\}.\end{equation*}
	(Since $d>0$, the above set is nonempty.)
	For $1\le i\le n$, we denote ${\boldsymbol \epsilon}_i$ as the vector in $\RR^n/{\RR \bf 1}$ whose $i$-th entry is $\epsilon$ and other entries are zero.
	Then we have
	\begin{equation}\label{equ:plus}
		d_{tr}({\bf u}+{\boldsymbol \epsilon}_i,{\bf v})=\begin{cases}
		d, &\text{ if } i\notin \peak({\bf u,v})\cup \valley({\bf u,v}); \\
		d+\epsilon, &\text{ if } i\in \peak({\bf u,v}); \\
		d-\epsilon, &\text{ if } i\in \valley({\bf u,v}) \text{ and } |\valley({\bf u,v})|=1; \\
		d &\text{ if } i\in \valley({\bf u,v}) \text{ and } |\valley({\bf u,v})|\ge 2. \\
		\end{cases}
	\end{equation}
	Similarly,
	\begin{equation}\label{equ:minus}
		d_{tr}({\bf u}-{\boldsymbol \epsilon}_i,{\bf v})=\begin{cases}
		d, &\text{ if } i\notin \peak({\bf u,v})\cup \valley({\bf u,v}); \\
		d+\epsilon, &\text{ if } i\in \valley({\bf u,v}); \\
		d-\epsilon, &\text{ if } i\in \peak({\bf u,v}) \text{ and } |\peak({\bf u,v})|=1; \\
		d &\text{ if } i\in \peak({\bf u,v}) \text{ and } |\peak({\bf u,v})|\ge 2. \\
		\end{cases}
	\end{equation}
\end{lemma}

\begin{proof}
	We use formula (\ref{equ:difset}). Let $\mathcal{D}_{\bf u,v}$ be the set $\{u_i-v_i|1\le i\le n\}$ for any two vectors ${\bf u,v}\in \RR^n/{\RR \bf 1}$. 
	
	We consider $d_{tr}({\bf u}+{\boldsymbol \epsilon}_i,{\bf v})$ first. If $i\notin peak({\bf u,v})\cup valley({\bf u,v})$, then $u_i-v_i$ is between the maximum and minimum of $\mathcal{D}_{\bf u,v}$. So $\mathcal{D}_{{\bf u}+{\boldsymbol \epsilon}_i,{\bf v}}$ has the same maximum and minimum as $D_{\bf u,v}$, then $d_{tr}({\bf u}+{\boldsymbol \epsilon}_i,{\bf v})=d_{tr}({\bf u,v})$. If $i\in peak({\bf u,v})$, then $\mathcal{D}_{{\bf u}+{\boldsymbol \epsilon}_i,{\bf v}}$ has the same minimum as $\mathcal{D}_{\bf u,v}$, but $\max(\mathcal{D}_{{\bf u}+{\boldsymbol \epsilon}_i,{\bf v}})=\max(\mathcal{D}_{\bf u,v})+\epsilon$. So $d_{tr}({\bf u}+{\boldsymbol \epsilon}_i,{\bf v})=d_{tr}({\bf u,v})+\epsilon$. 
	
	If $i\in valley({\bf u,v})$, then $u_i-v_i$ is the minimum of $\mathcal{D}_{\bf u,v}$. So $\mathcal{D}_{{\bf u}+{\boldsymbol \epsilon}_i,{\bf v}}$ has the same maximum as $\mathcal{D}_{\bf u,v}$. As for the minimum, if $|valley({\bf u,v})|\ge 2$, then there exists $k\ne i$ with $u_k-v_k=u_i-v_i$. Then $u_k-v_k\in \mathcal{D}_{{\bf u}+{\boldsymbol \epsilon}_i,{\bf v}}$ and $\mathcal{D}_{{\bf u}+{\boldsymbol \epsilon}_i,{\bf v}}$ has the same minimum as $\mathcal{D}_{\bf u,v}$. As a result, $d_{tr}({\bf u}+{\boldsymbol \epsilon}_i,{\bf v})=d_{tr}({\bf u,v})$. If $|valley({\bf u,v})|=1$, then all other elements in $\mathcal{D}_{\bf u,v}$ are strictly greater than $u_i-v_i$, thus $\min(\mathcal{D}_{{\bf u}+{\boldsymbol \epsilon}_i,{\bf v}})=\min(\mathcal{D}_{\bf u,v})+\epsilon$. So $d_{tr}({\bf u}+{\boldsymbol \epsilon}_i,{\bf v})=d_{tr}({\bf u,v})-\epsilon$.
	
	The cases of $d_{tr}({\bf u}-{\boldsymbol \epsilon}_i,{\bf v})$ could be analyzed in the same way.
\end{proof}

Next, for (\ref{equ:plus}) and (\ref{equ:minus}), if we sum over $i$, we get the following corollary.
\begin{corollary}\label{cor:vari}
	Let ${\bf u},{\bf v}$ be two points in $\RR^n/{\RR \bf 1}$. Let $d,e$ and ${\boldsymbol \epsilon}_i$ be the same as in Lemma \ref{lem:vari}. Then
	\begin{equation*}\mathop{\sum}_{i=1}^{n}{(d_{tr}({\bf u}+{\boldsymbol \epsilon}_i,{\bf v})+d_{tr}({\bf u}-{\boldsymbol \epsilon}_i,{\bf v}))}
	=2n\cdot d+[f(|peak({\bf u,v})|)+f(|valley({\bf u,v})|)]\cdot e,
	\end{equation*}
	where $f$ is the function defined on $\ZZ_{}+$ by
	\begin{equation*}f(n)=\begin{cases}
	0, &\text{ if } n=1; \\
	n, &\text{ if } n\ge 2. \\
	\end{cases}.\end{equation*}
\end{corollary}

\begin{definition}\label{defn:sim}
	Let $m$ and $n$ be positive integers. Two subsets $S,T \subset [m]\times [n]$ are called \emph{similar} if for $1\le i\le m$ we have
	\begin{equation*}|\{k|(i,k)\in S \}|=|\{k|(i,k)\in T \}|\end{equation*}
	and for $1\le j\le n$ we have
	\begin{equation*}|\{k|(k,j)\in S \}|=|\{k|(k,j)\in T \}|.\end{equation*}
	In other words, $S$ and $T$ are similar if and only if given any row or column of $M$, they have the same number of elements in it.	
\end{definition}

The following lemma explicitly tells us the defining equations of the finite union of proper linear subspaces. 

\begin{lemma}\label{lem:defequ}
	Let $X=(x_{i,j})_{1\le i\le m,1\le j\le n}$ be an $m\times n$ matrix. For any set $S\subset [m]\times [n]$, let 
	\begin{equation*}x_{S}=\mathop{\sum}_{(i,j)\in S}{x_{i,j}}.\end{equation*}
	If the row vectors of $X\in \RR^{m(n-1)}$ form an essential set with a unique Fermat-Weber point in $\RR^n/{\RR \bf 1}$, then there exist disjoint $S,T\subset [m]\times [n]$ such that $S$ and $T$ are similar and $x_{S}=x_{T}$.
\end{lemma}

\begin{proof}
	Suppose $X$ is an $m\times n$ matrix with entries $x_{i,j}$ such that the row vectors ${\bf v}_1,\ldots,{\bf v}_m$ of $X$ form an essential set of points in $\RR^n/{\RR \bf 1}$ with a unique Fermat-Weber point ${\bf c}\in \RR^n/{\RR \bf 1}$. Then the points ${\bf v}_1-{\bf c},\ldots,{\bf v}_m-{\bf c}$ also form an essential set and they have a unique Fermat-Weber point $\bf 0$. Let $X'=(x'_{i,j})$ be the corresponding matrix of these points. Then $x'_{i,j}=x_{i,j}-c_{j}$ for any $1\le i\le m,1\le j\le n$. Note that for $S,T\subset [m]\times [n]$, if $S$ and $T$ are similar, then $x_{S}=x_{T}$ if and only if $x'_{S}=x'_{T}$. Then we may assume the unique Fermat-Weber point of ${\bf v}_1,\ldots,{\bf v}_m$ is $\bf 0$.
	
	Now we construct an undirected graph $G=(V,E)$. Let $V=[n]$. For $1\le i\le m$, let $P_i=peak({\bf 0},{\bf v}_i)$ and $Q_i=valley({\bf 0},{\bf v}_i)$. Then $P_i,Q_i\subset [n]$. For $1\le i\le n$, we choose an arbitrary tree $T_{P_{i}}$ whose set of vertices is $P_i$ and include its edges into $E$ and we choose an arbitrary tree $T_{Q_{i}}$ whose set of vertices is $Q_i$ and include its edges into $E$. Note that if $|P_i|=1$ then $T_{P_{i}}$ has no edge. Here we allow parallel edges in $G$, because $P_i$ may equal to $P_j$ for different $i$ and $j$. 
	
	It suffices to show that $G$ contains a cycle. Suppose one minimal cycle in $G$ has $r$ distinct vertices $j_1,j_2,\ldots,j_r \in [n]$, where for each $1\le t\le r$ there is an edge connecting $j_t$ and $j_{t+1}$ (we denote $j_{r+1}=j_1$). By definition, there exists $i_t\in [m]$ such that $\{j_t,j_{t+1}\} \subset P_{i_t}$ or $\{j_t,j_{t+1}\} \subset Q_{i_t}$. In either case we have that
	\begin{equation}\label{equ:loc}
	x_{i_t,j_t}=x_{i_t,j_{t+1}}.
	\end{equation}
	Then we define the two subsets $S,T$ of $[m]\times [n]$ as follows:
	\begin{equation*}S=\{(i_t,j_t)|1\le t\le r\}\text{ , } T=\{(i_t,j_{t+1})|1\le t\le r\}.\end{equation*}
	Then $x_{S}=x_{T}$ follows from (\ref{equ:loc}). In addition, for $j\in [n]$, if $j=j_t$ for some $t$ then both $S$ and $T$ have one element in the $j$-th column of $M$; otherwise both $S$ and $T$ have no elements in the $j$-th column of $M$. For $i \in [m]$, both $S$ and $T$ have $|\{t|i_t=i\}|$ elements in the $i$-th row of $M$. Then $S$ and $T$ are similar. Next we show that $S\ne T$. Suppose $S=T$, then for each $1\le t\le r$, the unique element of $S$ in the $j_t$-th column is equal to the unique element of $T$ in the $j_t$-th column, which means $(i_t,j_t)=(i_{t-1},j_t)$. Then we have $i_t=i_{t-1}$.
	So $i_1=i_2=\ldots=i_r$, which means all $r$ vertices in the cycle are chosen from $P_i \cup Q_i$. Since $P_i$ and $Q_i$ are disjoint, either all vertices are chosen from $P_i$ or all vertices are chosen from $Q_i$. Then in either case, the edges in the cycles are either all chosen from $T_{P,i}$ or all chosen from $T_{Q,i}$, which contradicts the fact that both $T_{P,i}$ and $T_{Q,i}$ are trees. Therefore $S\ne T$. Finally if $S$ and $T$ have common elements, then we can delete them to get another pair of similar subsets $S',T'$, and we still have $x_{S'}=x_{T'}$. So we can choose disjoint $S$ and $T$.
	
	Finally we show that $G$ contains a cycle. We compute the following sum
	\begin{equation*}
	\mathcal{K}=\mathop{\sum}_{i=1}^{m}{\mathop{\sum}_{j=1}^{n}{(d_{tr}({\boldsymbol \epsilon}_j,{\bf v}_i)+d_{tr}(-{\boldsymbol \epsilon}_j,{\bf v}_i))}}.
	\end{equation*}
	On one hand, since ${\bf v}_1,\ldots,{\bf v}_m$ have a unique Fermat-Weber point $\bf 0$, we have
	\begin{equation}\label{ine:pos}
		\mathop{\sum}_{i=1}^{m}{d_{tr}({\bf w},{\bf v}_i)}>\mathop{\sum}_{i=1}^{m}{d_{tr}({\bf 0},{\bf v}_i)}
	\end{equation}
	for any nonzero vector ${\bf w}\in \RR^n/{\RR \bf 1}$. By (\ref{equ:plus}) and (\ref{equ:minus}), for $1\le i\le m$
	\begin{equation*}d_{tr}(\pm{\boldsymbol \epsilon}_j,{\bf v}_i)-d_{tr}({\bf 0},{\bf v}_i)\end{equation*}
	is $\pm \bf \epsilon$ or zero. Then difference between the LHS and the RHS of (\ref{ine:pos}) is an integer multiple of $\bf \epsilon$. Hence
	\begin{equation}\label{ine:str}
		\mathop{\sum}_{i=1}^{m}{d_{tr}({\bf w},{\bf v}_i)}-\mathop{\sum}_{i=1}^{m}{d_{tr}({\bf 0},{\bf v}_i)}\ge \epsilon
	\end{equation}
	for ${\bf w}=\pm {\boldsymbol \epsilon}_j$. Summing over $j$, we have
	\begin{equation}\label{ine:lowbd}
		\mathcal{K}\ge 2n\mathop{\sum}_{i=1}^{m}{d_{tr}({\bf 0},{\bf v}_i)}+2n\cdot \epsilon.
	\end{equation}
	
	On the other hand, by Corollary \ref{cor:vari} we have
	\begin{equation*}\mathop{\sum}_{j=1}^{n}{(d_{tr}({\boldsymbol \epsilon}_j,{\bf v}_i)+d_{tr}(-{\boldsymbol \epsilon}_j,{\bf v}_i))}=2n\cdot d_{tr}({\bf 0},{\bf v}_i)+[f(|P_i|)+f(|Q_i|)]\cdot \epsilon.\end{equation*}
	Summing over $i$ we have
	\begin{equation}\label{equ:equ}
		\mathcal{K}=2n\mathop{\sum}_{i=1}^{m}{d_{tr}({\bf 0},{\bf v}_i)}+\left[\mathop{\sum}_{i=1}^{m}{(f(|P_i|)+f(|Q_i|))}\right] \cdot \epsilon.
	\end{equation}
	Comparing (\ref{ine:lowbd}) and (\ref{equ:equ}), we get
	\begin{equation}\label{ine:sum}
	    \mathop{\sum}_{i=1}^{m}{(f(|P_i|)+f(|Q_i|))}\ge 2n.
	\end{equation}
    Next, for $x\ge 2$ we have
    \begin{equation*}x-1\ge \frac{x}{2}=\frac{1}{2}f(x)\end{equation*}
    and when $x=1$ both $x-1$ and $f(x)$ are zero. Then
    \begin{equation*}
    \mathop{\sum}_{i=1}^{m}{(|P_i|-1)+(|Q_i|-1)}\ge \frac{1}{2}\mathop{\sum}_{i=1}^{m}{(f(|P_i|)+f(|Q_i|))}\ge n.
    \end{equation*}
    So the graph $G$ has at least $n$ edges and it contains a cycle.
\end{proof}

\begin{proof}[Proof of Theorem \ref{thm:lowdim}]
	For $(i,j)\in [m]\times [n]$, let $X_{i,j}$ be variables. For $S\subset [m]\times [n]$, let
	\begin{equation*}X_{S}=\sum_{(i,j)\in S}{X_{i,j}}.\end{equation*}
	We define the polynomial
	\begin{equation*}F=\mathop{\prod}_{\substack{S,T\subset [m]\times [n] \\ S\ne T\\ S,T \text{ are similar}} }{(X_{S}-X_{T})}.\end{equation*}
	Then $F\in \RR[X_{1,1},X_{1,2},\ldots,X_{m,n}].$
	By Lemma \ref{lem:defequ}, if an $m\times n$ matrix
        $M=(m_{i,j})$ corresponds to an essential set of $m$ points
        with a unique Fermat-Weber point in $\RR^n/{\RR \bf 1}$, then
        there exist distinct $S,T\subset [m]\times [n]$ such that $S$
        and $T$ are similar and $M_S=M_T$. So $F((m_{i,j}))=0$. As a
        result, the points of $\RR^{m(n-1)}$ corresponding to an
        essential set of $m$ points with a unique Fermat-Weber point
        in $\RR^n/{\RR \bf 1}$ are contained in 
the union of proper linear subspaces $V(F)$.
\end{proof}

The immediate consequence of Theorem \ref{thm:lowdim} is as follows:
\begin{corollary}\label{cor:open}
	If we choose a random sample in the moduli space $\RR^{m(n-1)}$
        with any distribution $\nu$ with $\nu(L) = 0$ for
        any $L \subset \RR^{m(n-1)}$ with the dimension
        of $L$ is strictly less than $m(n-1)$, then we have
        probability $1$ to get either a random sample that is not
        essential, or a random sample that has more than one (thus infinitely many) Fermat-Weber points.
\end{corollary}

\begin{proof}
	Let $C$ be the random sample in $\RR^{m(n-1)}$ that
        corresponds to an essential set of $m$ points with a unique
        Fermat-Weber point. Then it suffices to show that the measure
        of $C$ is zero. By Theorem \ref{thm:lowdim}, $C$ is contained
        in the finite union of hypersurfaces $V(X_{S}-X_{T})$, where
        $S,T\subset [m]\times [n]$, $S$ and $T$ are distinct and
        similar. Then for each pair of such $S$ and $T$, the
        hypersurface $V(X_{S}-X_{T})$ is isomorphic to
        $\RR^{m(n-1)-1}$. So it has measure zero. Thus, the measure of
        this finite union is still zero, and so is $C$.	
\end{proof}

\begin{definition}[Tropical Determinant]
	Let $X=(x_{ij})$ be an $n \times n$ matrix with real entries. Then its tropical determinant is defined as follows:
	\begin{equation}\label{equ:det}
		\text{trop}\det X = \min_{\pi\in S_n} {\sum_{i=1}^{n}{x_{i\pi(i)}}}.
	\end{equation}
	A matrix $X$ is tropically singular if the minimum is attained at least twice in (\ref{equ:det}).
\end{definition}

In the proof of Lemma \ref{lem:defequ}, the subsets $S$ and $T$
are very similar to the terms in the tropical determinant of matrices. However the following example shows that the
matrix does not need to have a minor whose tropical determinant contains
two equal terms.

\begin{example}\label{ex:perm}[No equal terms in the tropical determinant of all minors]
	The following five points in $\RR^3/{\RR \bf 1}$ 
	\begin{equation*}(1, -1, -1), (-1, 1, -1), (1, 1, -1), (0, -1, 1), (-1, 0, 1)\end{equation*}
	form an essential set and they have a unique Fermat-Weber
        point $(0,0,0)$ in $\RR^3/{\RR \bf 1}$. However, let $M$ be
        the corresponding $5\times 3$ matrix.  No minor of $M$ is
        tropically singular.
In addition, for every minor of $M$, its tropical determinant has no equal terms. 
\end{example}

\begin{remark}\label{rem:ne}
	The converse of Theorem \ref{thm:lowdim} is not true in general. The following three points in $\RR^4/{\RR \bf 1}$
	\begin{equation*}(0,0,0,5),(0,0,3,1),(0,4,5,7)\end{equation*}
	correspond to a point in the finite union of proper linear subspaces
of $\RR^{9}$ as in Lemma \ref{lem:defequ}, because we can take $S=\{(1,1),(2,2)\}$ and $T=\{(1,2),(2,1)\}$. However, their polytope of Fermat-Weber point is a line segment in $\RR^4/{\RR \bf 1}$ with endpoints
	\begin{equation*}(0,2,3,5),(0,3,3,5).\end{equation*}
	So these three points form an essential set and they have more than one Fermat-Weber points.
\end{remark}

\section{The Fermat-Weber points within treespaces}\label{sec:tree}

In this section we focus on the space of phylogenetic trees. 
An {\em equidistant tree} is a weighted rooted phylogenetic tree whose distance from
the root to each leaf is the same real number for all its leaves.  
Suppose $\Tn$ is the space of all equidistant trees with $N$
leaves, i.e., the set of leaves is $\{1, 2, \ldots, N\}$. For positive integer $N$, we denote by $[N]$ the set $\{1,2,\ldots,N\}$.

\begin{definition}       
The \textit{distance} $D_{ij}(T)$, between two leaves $i$ and $j$ in
$T \in \Tn$, is the length of a unique path between leaves $i$ and $j$.
 The \textit{distance matrix} of $T \in \Tn$ is a $N \times N$ matrix
 $D(T)=(D_{ij})_{1\le i,j\le N}$ 
       $\forall i,j$ $(1 \le i , j \leq N)$, where $N$ is the number of leaves in the tree $T$. The \textit{metric} of $T\in \Tn$, denoted by $D=(D_{ij})_{1\le i<j\le N}$, is a vector with $\binom{n}{2}$ entries.
\end{definition}

Distance matrices of equidistant trees in $\Tn$ satisfy the 
following strengthening of the triangle inequalities:
\begin{equation}
 \label{eq:ultrametric} \quad
D_{ik} \leq  {\rm max}(D_{ij} , D_{jk} )
\quad \hbox{for all} \,\,\, i,j,k \in [N]. 
\end{equation}
If \eqref{eq:ultrametric} holds, then the metric $D$ is called an {\em ultrametric}.
The set of all  ultrametrics contains the
ray $ \mathbb{R}_{\geq 0} {\bf 1}$ spanned by the all-one metric ${\bf 1}$,
defined by $D_{ij} = 1$ for $1 \leq i < j \leq N$.
The image of the set of ultrametrics in the quotient 
space $\mathbb{R}^{\binom{N}{2}} \!/ \mathbb{R} {\bf 1}$
is called the {\em space of ultrametrics}.  This is the image of ultrametrics
in the quotient space using the extrinsic metric, via the tropical metric \cite{BL}.

Suppose
we have a set of equidistant phylogenetic trees with $N$ leaves. They are
represented by their metrics $D$ in $\RR^{\binom{N}{2}}$ so that the space of equidistant
phylogenetic trees $\Tn$ with fixed number of leaves $N$ can be represented
by a union of polyhedra in $\RR^{\binom{N}{2}}$.  In the previous
sections, we have shown that there might be infinitely many
Fermat-Weber points of them. However, many of those points may not
correspond to any phylogenetic tree. In this section, for a sample of points in $\Tn$, we consider the set of their Fermat-Weber points within $\Tn$.
 
The spaces of equidistant phylogenetic trees $\Tn$ with $N$ leaves have $(2N-3)!!$
maximal polyhedra with dimension $N-2$ \cite{BHV,Sturmfels2004,LSTY}. The intersection of each maximal polyhedron and the polytope
of Fermat-Weber points is either empty or a polytope. Here we
investigate the set of equidistant phylogenetic trees such that they form an
essential set and there exists a unique equidistant phylogenetic tree that is a
Fermat-Weber point of them. 

We conducted simulations on 
Fermat-Weber points of a sample in $\Tn$ for $N = 4$.  We generated
$60$ equidistant phylogenetic trees with $N = 4$ leaves
using the {\tt R} package {\tt ape} \cite{ape}.   Due to the computational
time, we set $60$ as a sample size.
 Among these $60$ trees, we sampled randomly subsets of sizes
$4$, $5$ and $6$. For each subsample, we computed its Fermat-Weber
points within treespaces by using {\tt Maple}${}^\mathrm{TM}$ 2015
\cite{Maple2015}. We counted the maximal dimension of the set of 
Fermat-Weber points, which is a finite union of classical convex polytopes in
$\RR^{5}$ by Proposition \ref{prop:polytope}.  The result is shown in Table \ref{tab:maxdim}.

\begin{table}[h]
	\centering
	\begin{tabular}{|l|c|c|c|}
		\hline
		Sample size \textbackslash Max Dim. & $0$ & $1$ & $2$ \\
		\hline
		$4$ & $2$ & $7$ & $51$ \\
		\hline
		$5$ & $6$ & $15$ & $39$ \\
		\hline
		$6$ & $10$ & $21$ & $29$ \\
		\hline
	\end{tabular}
	\caption{The maximal dimension of the set of Fermat-Weber
          points within the treespace: Samples with size $4$, $5$, or $6$ phylogenetic trees with $4$ leaves.}\label{tab:maxdim}
\end{table}

\begin{example}\label{ex:uni}
	The polytope of Fermat-Weber points of the following four trees with $4$ leaves 
\begin{equation*}	
       \begin{split}
	&(32/109, 1, 124/673, 1, 32/109, 1), (1, 6/85, 1, 1, 203/445, 1),\\
	& (1, 1, 1, 310/783, 310/783, 1/265), (47/510, 1, 1, 1, 1, 125/151).
	\end{split}
\end{equation*}	
	is $2$-dimensional, while there is a unique Fermat-Weber point that corresponds  a phylogenetic tree, which is $(1,1,1,1,1,1)$.
\end{example}
We have the following conjecture based on our simulations.
\begin{conjecture}
	A sample in $\Tn$ like in Example \ref{ex:uni} is the only case of a unique tree as
        Fermat-Weber point. In other words, if a sample in $\Tn$ has a
        unique Fermat-Weber point, then 
        its unique Fermat-Weber point is the all-one vector ${\bf 1}$.
\end{conjecture}

\begin{remark}
	We have tried to conduct similar experiments for $N\ge 5$
        but the computational time was not feasible. The computational
        time complexity for our simulation study does not come from the number of polyhedra in the treespace, but comes from the difficulty of computing the polytope of Fermat-Weber points. See Section \ref{sec:dis} for details.
\end{remark}

\section{The $k$-ellipses under the tropical metric}\label{sec:elli}

Let $k$ be a positive integer. Given a sample of $k$ points $\vv_1,\vv_2,\ldots,\vv_k \in \RR^n/\RR{\bf 1}$, the locus of points $\uu \in \RR^n/\RR{\bf 1}$ such that
\begin{equation}\label{equ:locus}
	\mathop{\sum}_{i=1}^{k}{d_{tr}(\uu, \vv_i)}={\bf d}(\vv_1,\vv_2,\ldots,\vv_k)
\end{equation}
is the polytope of Fermat-Weber points of $\vv_1,\vv_2,\ldots,\vv_k$. In this section we generalize this locus and discuss the $k$-ellipses under the tropical metric.

\begin{definition}\label{def:elli}
	Let $\vv_1,\vv_2,\ldots,\vv_k \in \RR^n/\RR{\bf 1}$ and $a\ge {\bf d} (\vv_1,\vv_2,\ldots,\vv_k)$. Then the $k$-ellipse with foci $\vv_1,\vv_2,\ldots,\vv_k$ and mean radius $\frac{a}{k}$ is the follow set of points in $\RR^n/\RR{\bf 1}$:
	\begin{equation}
		\{\uu\in \RR^n/\RR{\bf 1}|\mathop{\sum}_{i=1}^{k}{d_{tr}(\uu, \vv_i)}=a \}.
	\end{equation}
\end{definition}

\begin{proposition}\label{prop:elli}
	Let $\vv_1,\vv_2,\ldots,\vv_k \in \RR^n/\RR{\bf 1}$ and $a\ge {\bf d} (\vv_1,\vv_2,\ldots,\vv_k)$. Then the $k$-ellipse with foci $\vv_1,\vv_2,\ldots,\vv_k$ and mean radius $\frac{a}{k}$ is a classical convex polytope in $\RR^{n-1}\simeq \RR^n/\RR{\bf 1}$.
\end{proposition}

\begin{proof}
	The proof is very similar to the one of Proposition \ref{prop:polytope}. Note that we can still eliminate the parameters $c_i$, and now the inequalities in (\ref{ine:lin}) become
	\begin{equation}\label{ine:elli}
		\sum_{i=1}^{m}{(x_{j_i}-x_{k_i}+v_{i,k_i}-v_{i,j_i})}\le a.
	\end{equation}
	So this $k$-ellipse is also a polyhedron in $\RR^{n-1}$ and for the same reason it is bounded.
\end{proof}

\begin{example}
	We consider Example \ref{example:threepoints} again.  Let
        $\vv_1 = (0, 0, 0), \vv_2= (0, 3, 1), $\\$ \vv_3= (0, 2, 5)$. Then by Theorem \ref{thm:sum}, we have ${\bf d} (\vv_1,\vv_2,\vv_3)=(0+3+5)-(0+1+0)=7$. We consider $a=8,10,50,100$. Figure \ref{ex:ellipse} shows the $3$-ellipses with foci $\vv_1,\vv_2,\vv_3$ and mean radius $\frac{a}{3}$.
\end{example}

\begin{figure}[h]
	\centering
	\begin{minipage}[t]{0.45\textwidth}
		\centering
		\includegraphics[width=0.45\textwidth]{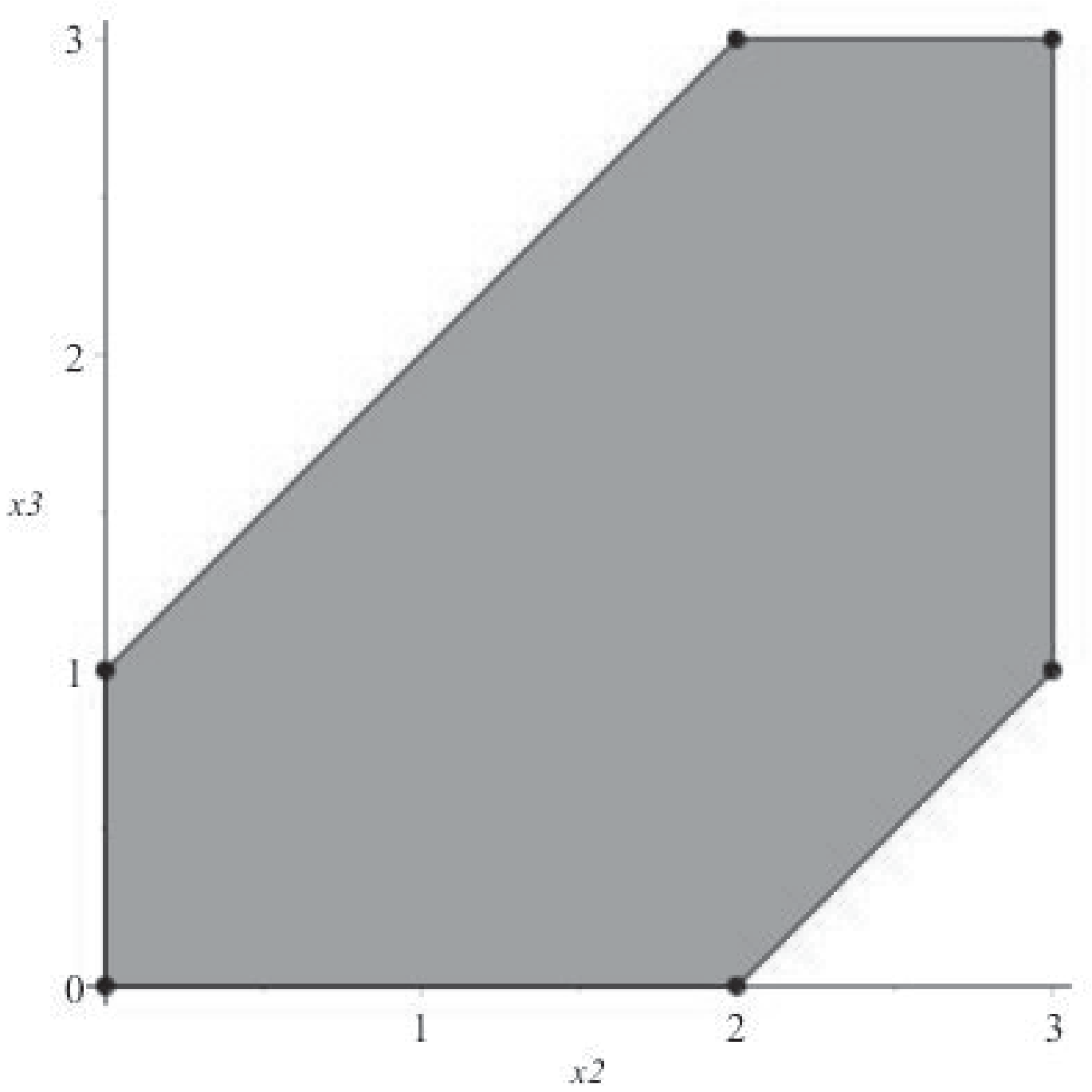}
		\centerline{The $3$-ellipse with $a=8$. It is a hexagon.}
	\end{minipage}
	\hspace{1cm}
	\begin{minipage}[t]{0.45\textwidth}
		\centering
		\includegraphics[width=0.45\textwidth]{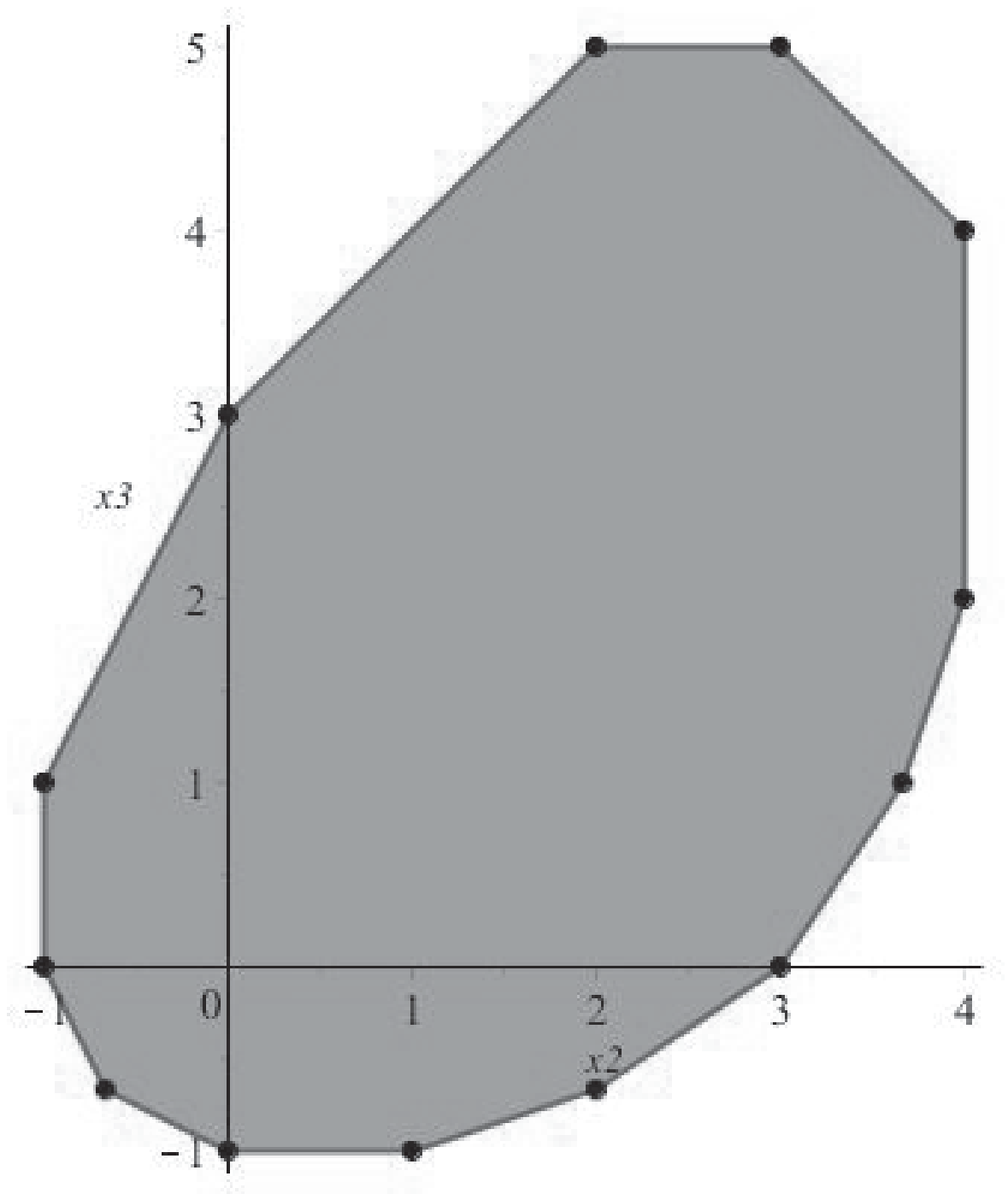}
		\centerline{The $3$-ellipse with $a=10$. It is a $13$-gon.}
	\end{minipage}
	
	\begin{minipage}[t]{0.4\textwidth}
		\centering
		\includegraphics[width=0.45\textwidth]{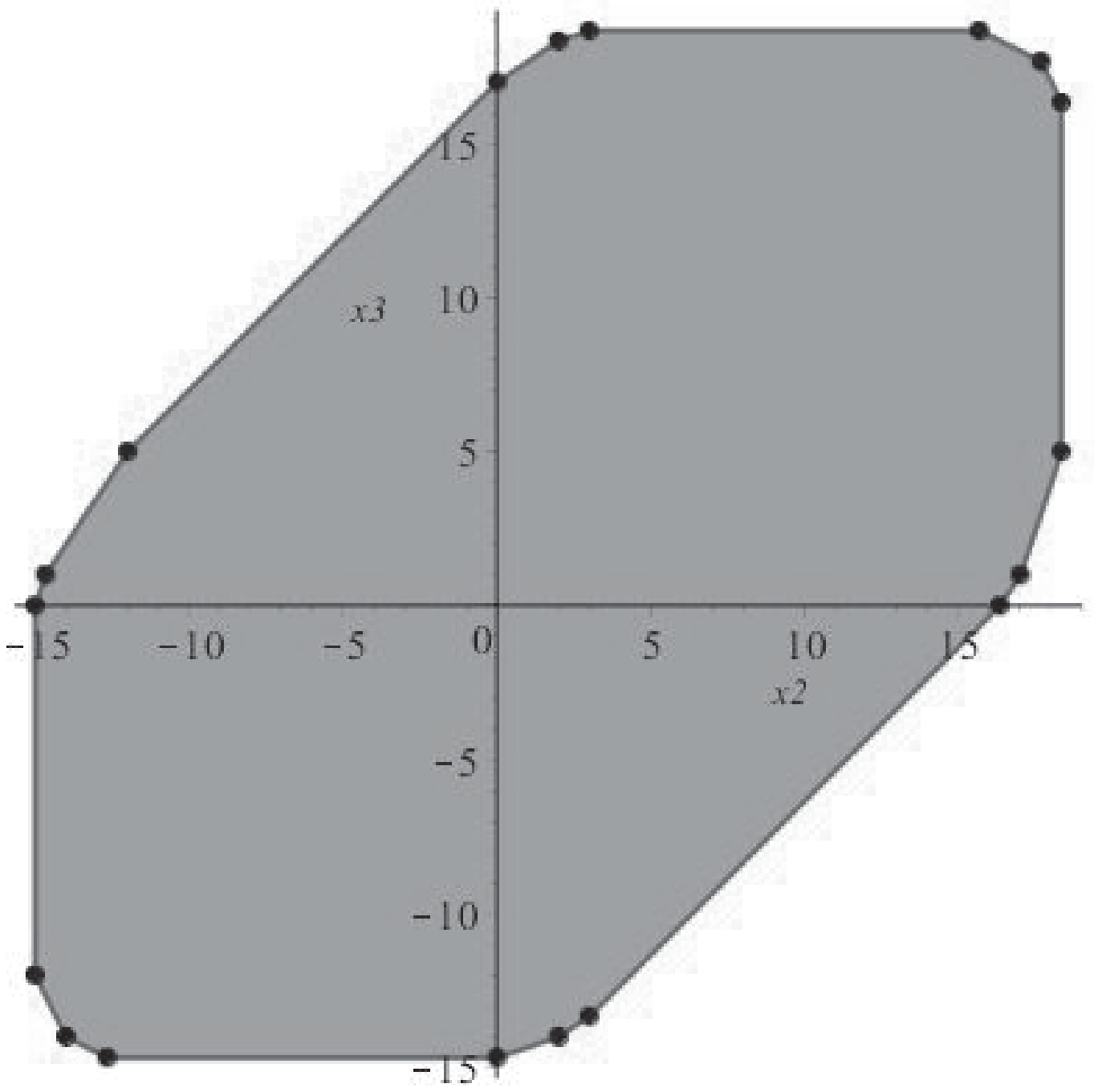}
		\centerline{The $3$-ellipse with $a=50$. It is an $18$-gon.}
	\end{minipage}
	\hspace{2cm}
	\begin{minipage}[t]{0.4\textwidth}
		\centering
		\includegraphics[width=0.45\textwidth]{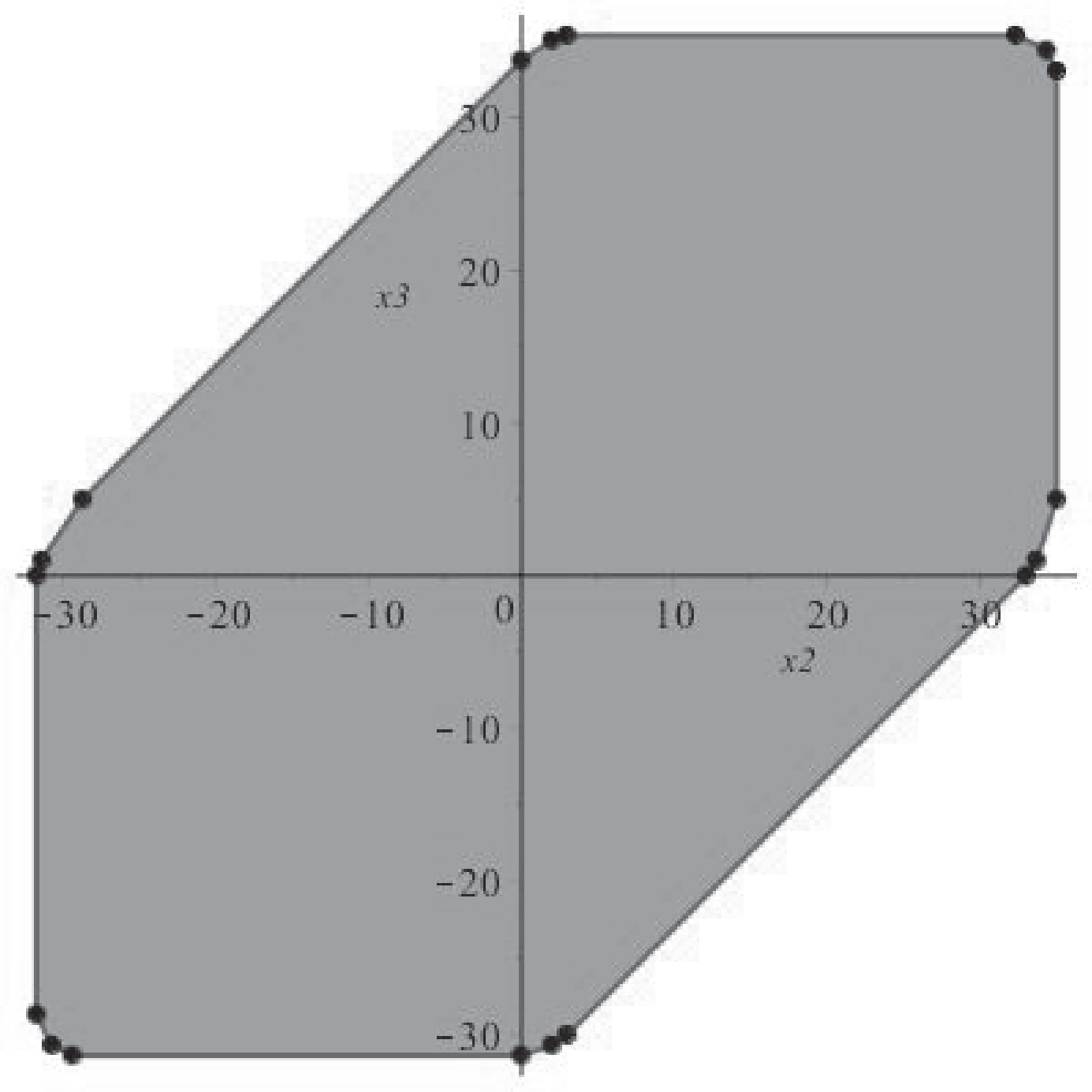}
		\centerline{The $3$-ellipse with $a=100$. It is an $18$-gon.}
	\end{minipage}
	\caption{Four $3$-ellipses with foci $\vv_1,\vv_2,\vv_3$ and different mean radii.}\label{ex:ellipse}
\end{figure}

\section{Computing the Fermat-Weber points under the tropical metric}\label{sec:dis}

In this section we explain our method of computing the set of all Fermat-Weber points of a sample and discuss some computational issues. 
Suppose points in a sample are $\vv_1,\vv_2,\ldots,\vv_m \in \RR^n/\RR{\bf
  1}$. Then our method consists of two steps:
\begin{enumerate}[(a)]
	\item\label{stepa} to compute ${\bf d}={\bf d} (\vv_1,\vv_2,\ldots,\vv_m)$;
	\item\label{stepb} given $d$, to compute the set of Fermat-Weber points of $\vv_1,\vv_2,\ldots,\vv_m$.
\end{enumerate}

One way to compute step \eqref{stepa} is to use Theorem \ref{thm:sum}. But then we have to compute all possible functions $\sigma,\tau:
[m]\to [n]$ such that $\sigma([m])=\tau([m])$ as multisets. The number of such functions is $n^m$, which means the time complexity of this step would be exponential in $m$. In practice we use a method of linear programming, which minimizes $d$ such that all inequalities in (\ref{ine:1dis}) and $d\ge \sum_{i=1}^{m}{c_i}$ are feasible simultaneously.

However, for step \eqref{stepb}, even if we have $d$, there are still
many inequalities that define the polytope of Fermat-Weber
points. In the proof of Proposition
\ref{prop:polytope}, if we eliminated the parameters $c_i$, then there
are $\binom{n}{2}^{m}$ inequalities in (\ref{ine:lin}); otherwise we
may keep the parameters $c_i$ and get another polytope in the ambient
space $\RR^{n+m-1}$ and then project it to $\RR^{n-1}$, but for each
$c_i$ there are still $2\binom{n}{2}=n(n-1)$ inequalities, so we need
$mn(n-1)+1$ inequalities to define this polytope.  
From the computations
with {\tt polymake} \cite{polymake}, there seem to be some redundant
inequalities but we do not know an efficient method for step
\eqref{stepb}.  Note that we used {\tt polymake} since this software
is one of the most efficient software to deal with polyhedral
geometry.
In this paper, the time complexity of our computation of all
Fermat-Weber points from a given sample is not very
efficient.  But still we do not know the computational time
complexity, i.e., finding a tropical Fermat-Weber point of the given
sample is not known.

\begin{question} 
What is the time complexity to compute the set of tropical Fermat-Weber points
of a sample of $m$ points in $\RR^n/\RR{\bf 1}$ in $m$ and $n$?  Is there a
polynomial time algorithm to compute the vertices of the polytope of tropical Fermat-Weber points
of a sample of $m$ points in $\RR^n/\RR{\bf 1}$ in $m$ and $n$?
\end{question}

Note that the linear system which we present here has many redundant
inequalities because we simply followed the definition, without any
simplification. This leads to the following question:
\begin{question}
In terms of polyhedral geometry, what are the facets defining
a polytope for all tropical  Fermat-Weber points of a given sample
as well as the number of the facets of the polytope, i.e., the number of the 
minimal set of inequalities needed to define the set of all 
tropical Fermat-Weber points of a given sample?
\end{question}

As we have discussed above, it is very hard to compute the set of all
tropical Fermat-Weber points over treespaces because of its computational time.  At this
moment, we can compute the Fermat-Weber points on treespaces of at most
$4$ leaves.  
As future research projects, it will
be interesting to compare the set of all
tropical Fermat-Weber
points with summary/consensus trees,
such as the majority-rules consensus tree as well as the Fr\'{e}chet mean
over treespaces.

\section*{Acknowledgment}
B.L. is partly supported Coleman Fellowship for this work.  R.Y. is
partially supported NSF Division of Mathematical Sciences:
CDS\&E-MSS program 1622369. The authors also would like to thank Bernd Sturmfels for his comments on this work.

\bibliographystyle{siamplain}
\bibliography{refs}
\end{document}